\documentclass[11pt]{article}

\usepackage[a4paper,margin=1.1in]{geometry}
\usepackage{amsmath,amssymb,amsthm,mathtools}
\usepackage{enumitem}
\usepackage{hyperref}
\hypersetup{
  hidelinks,
  pdftitle={Adjoint closures of singular quadratic pencils and First's Pfister-type conjecture},
  pdfauthor={Shisong Xu},
  pdfkeywords={quadratic forms, systems of quadratic forms, Pfister local-global principle, adjoint closure, singular quadratic pencils, Kronecker blocks, minimal indices}
}
\usepackage{aliascnt}
\usepackage[nameinlink,capitalise]{cleveref}

\newtheorem{theorem}{Theorem}[section]

\newaliascnt{corollary}{theorem}
\newtheorem{corollary}[corollary]{Corollary}
\aliascntresetthe{corollary}

\newaliascnt{proposition}{theorem}
\newtheorem{proposition}[proposition]{Proposition}
\aliascntresetthe{proposition}

\newaliascnt{lemma}{theorem}
\newtheorem{lemma}[lemma]{Lemma}
\aliascntresetthe{lemma}

\theoremstyle{remark}
\newaliascnt{remark}{theorem}
\newtheorem{remark}[remark]{Remark}
\aliascntresetthe{remark}

\theoremstyle{definition}
\newaliascnt{definition}{theorem}
\newtheorem{definition}[definition]{Definition}
\aliascntresetthe{definition}

\newaliascnt{example}{theorem}

\aliascntresetthe{example}

\crefname{theorem}{Theorem}{Theorems}
\Crefname{theorem}{Theorem}{Theorems}
\crefname{proposition}{Proposition}{Propositions}
\Crefname{proposition}{Proposition}{Propositions}
\crefname{lemma}{Lemma}{Lemmas}
\Crefname{lemma}{Lemma}{Lemmas}
\crefname{corollary}{Corollary}{Corollaries}
\Crefname{corollary}{Corollary}{Corollaries}
\crefname{remark}{Remark}{Remarks}
\Crefname{remark}{Remark}{Remarks}
\crefname{definition}{Definition}{Definitions}
\Crefname{definition}{Definition}{Definitions}
\crefname{example}{Example}{Examples}
\Crefname{example}{Example}{Examples}
\crefname{section}{Section}{Sections}
\Crefname{section}{Section}{Sections}

\DeclareMathOperator{\End}{End}
\DeclareMathOperator{\Hom}{Hom}
\DeclareMathOperator{\Mat}{M}
\DeclareMathOperator{\Span}{span}
\DeclareMathOperator{\Tr}{Tr}
\DeclareMathOperator{\sgnop}{sgn}
\DeclareMathOperator{\rad}{rad}
\DeclareMathOperator{\Vis}{Vis}
\DeclareMathOperator{\Jac}{Jac}
\DeclareMathOperator{\Cent}{Cent}

\newcommand{\Cl}{\operatorname{Cl}}
\newcommand{\cQ}{\mathcal Q}
\newcommand{\cR}{\mathcal R}
\newcommand{\cS}{\mathcal S}
\newcommand{\cH}{\mathcal H}
\newcommand{\bQ}{\mathbb Q}
\newcommand{\bR}{\mathbb R}
\newcommand{\op}{\mathrm{op}}
\newcommand{\totalSgn}{\operatorname{sgn}}

\title{Adjoint closures of singular quadratic pencils and First's Pfister-type conjecture}
\author{Shisong Xu\\[0.35em]
\small Department of Mathematics, Nanjing University\\
\small 22 Hankou Road, Nanjing, Jiangsu 210093, People's Republic of China\\
\small Email: \href{mailto:shsxu@smail.nju.edu.cn}{shsxu@smail.nju.edu.cn}}
\date{}

\begin{document}
\maketitle

\begin{abstract}
We construct, over every formally real field, a singular pair of quadratic forms in dimension seven whose adjoint closure equals its pencil and consists entirely of hyperbolic forms in First's sense, although the pair is not weakly hyperbolic. This disproves First's conjecture that his local--global criterion for nonsingular pairs extends to singular pairs. The example has minimal dimension over formally real number fields and real closed fields. The construction uses an explicit closure formula: adjoining a symmetric singular Kronecker block of positive minimal index to a pair with a nondegenerate first member forces the adjoint closure of the sum to equal its pencil. We determine the corresponding closure formulas for Kronecker decompositions and compute the adjoint algebra, its Jacobson radical, and the involution-trace form of the seven-dimensional example.
\end{abstract}



\section{Introduction}
\label{sec:introduction}

Pfister's local--global principle identifies the torsion classes in the Witt group of a field of characteristic different from $2$ as those with zero signature at every ordering; see \cite[Chapter~VIII]{Lam2005}. First extended this principle to finite systems of quadratic forms \cite{First2020}. His criterion expresses weak hyperbolicity of a system $\cQ$ in terms of the involution-trace form of its adjoint algebra $A(\cQ)$. The construction uses the theory of adjoints and Hermitian categories developed in \cite{BFFM2014,QSS1979,Wilson2013}. For a nonsingular pair $\cQ=\{q_1,q_2\}$ over a number field or a real closed field, First proved the more specific criterion
\begin{equation}
\label{eq:first-B-intro}
 \cQ\text{ is weakly hyperbolic}
 \quad\Longleftrightarrow\quad
 \totalSgn q=0\text{ for every }q\in\Cl(\cQ).
\end{equation}
Here $\Cl(\cQ)$ is the space of quadratic forms satisfying every adjoint relation of $\cQ$, and $\totalSgn q=0$ means that $q$ has signature zero at every ordering. First conjectured that the nonsingularity hypothesis could be removed \cite[Introduction, after Theorem~B]{First2020}. The following theorem disproves this conjecture.

\begin{theorem}[Counterexample to First's conjecture]
\label{thm:main-intro}
Let $K$ be a formally real field. There is a singular pair
\[
 \cQ=\{Q_1,Q_2\}
\]
of quadratic forms on $K^7$ such that
\[
 \Cl(\cQ)=\Span_K\{Q_1,Q_2\},
\]
every form in $\Cl(\cQ)$ is hyperbolic, and $\cQ$ is not weakly hyperbolic. Consequently, the nonsingularity hypothesis in \eqref{eq:first-B-intro} cannot be omitted, even over $\bQ$ or a real closed field.
\end{theorem}

We use First's definition of hyperbolicity, which allows degenerate forms; see \cref{def:hyperbolic-systems}. The example is the orthogonal sum of his four-dimensional nonsingular pair \cite[Example~2.2]{First2020} and the three-dimensional symmetric singular Kronecker block $M_1$. Every form in the regular pencil is hyperbolic, but its adjoint closure contains a form of nonzero signature. Cross-adjoints with $M_1$ exclude this additional form from the closure of the sum. Since $M_1$ is hyperbolic as a system, adjoining it does not change weak hyperbolicity.

The closure calculation extends to every positive singular minimal index. Let
\[
 M_\varepsilon=\{S_{\varepsilon,1},S_{\varepsilon,2}\}
\]
be the symmetric singular Kronecker block of minimal index $\varepsilon\geq1$. We prove
\[
 \Cl(M_\varepsilon)=\Span_K\{S_{\varepsilon,1},S_{\varepsilon,2}\}.
\]
If $\cR=\{R_1,R_2\}$ and $R_1$ is nondegenerate, then
\begin{equation}
\label{eq:intro-collapse}
 \Cl(\cR\perp M_\varepsilon)
 =
 \Span_K\{R_1\perp S_{\varepsilon,1},
 R_2\perp S_{\varepsilon,2}\}.
\end{equation}
Thus adjoining one positive-minimal-index block forces the adjoint closure to equal the pencil.

For a Kronecker decomposition over an infinite field,
\[
 \cQ\cong \cR\perp M_{\varepsilon_1}\perp\cdots\perp M_{\varepsilon_s},
\]
with $s\geq1$ and $\cR=\{R_1,R_2\}$ nonsingular, let $\Vis_{\cR}(\cQ)$ denote the image of $\Cl(\cQ)$ under restriction to the regular summand. We prove that
\[
 \Vis_{\cR}(\cQ)=
 \begin{cases}
 \Cl(\cR),&\varepsilon_1=\cdots=\varepsilon_s=0,\\[1mm]
 \Span_K\{R_1,R_2\},&\max_i\varepsilon_i\geq1.
 \end{cases}
\]
If $R_1$ and $R_2$ are linearly independent, the same argument determines the entire closure. Writing $\cQ=\{Q_1,Q_2\}$, we obtain
\begin{equation}
\label{eq:intro-full-dichotomy}
 \Cl(\cQ)=
 \begin{cases}
 \{T_R\perp0:T_R\in\Cl(\cR)\},
 &\varepsilon_1=\cdots=\varepsilon_s=0,\\[1mm]
 \Span_K\{Q_1,Q_2\},
 &\max_i\varepsilon_i\geq1.
 \end{cases}
\end{equation}
Zero blocks therefore preserve the regular closure, whereas a positive minimal index restricts its image to the regular pencil.

These formulas also give the minimal dimension of a counterexample.

\begin{theorem}[Minimal dimension]
\label{thm:optimal-intro}
Let $K$ be a number field or a real closed field. Suppose that a singular pair $\cQ=\{q_1,q_2\}$ on $V$ satisfies
\[
 \totalSgn T=0
 \qquad\text{for every }T\in\Cl(\cQ),
\]
but $\cQ$ is not weakly hyperbolic. Then
\[
 \dim_K V\geq7.
\]
For formally real number fields and real closed fields, the bound is sharp.
\end{theorem}

The lower bound is the sum of two minima: four for a nonsingular pair that is not weakly hyperbolic despite having zero signatures throughout its pencil, and three for a singular Kronecker block of positive minimal index. For the seven-dimensional example, we also compute the adjoint algebra. It has dimension $18$, a $14$-dimensional Jacobson radical, semisimple quotient $K^4$, and involution-trace form
\[
 3x^2+7u^2+8ps.
\]
This form has signature $2$ at every ordering, giving a second proof that the pair is not weakly hyperbolic.

The classical classification of symmetric matrix pencils separates a regular part from singular Kronecker blocks \cite{LeepSchueller1999,Waterhouse1976}. The formulas above describe the adjoint closure of this decomposition, rather than the decomposition itself. Their proof is explicit: cross-adjoints constructed from $R_1^{-1}R_2$ force the regular restriction of each closure form to lie in the original pencil.

Section~\ref{sec:preliminaries} records the definitions and results used in the proofs. Section~\ref{sec:construction} constructs the counterexample. Section~\ref{sec:visibility} proves the general closure formulas and the dimension bound. Section~\ref{sec:adjoint-obstructions} computes the adjoint algebra and its involution-trace form.

\section{Systems, adjoints, and closures}
\label{sec:preliminaries}

Throughout, $K$ is a field of characteristic different from $2$, and all vector spaces are finite-dimensional over $K$. We write $\Mat_{m\times n}(K)$ for the space of $m\times n$ matrices, $\Mat_n(K)=\Mat_{n\times n}(K)$, and $I_V$ for the identity on $V$. We use the same symbol for a quadratic form $q$ and the associated symmetric bilinear form
\[
 (x,y)\longmapsto\tfrac12\bigl(q(x+y)-q(x)-q(y)\bigr).
\]
After choosing a basis, we identify $q$ with its symmetric Gram matrix $T$, so that $q(x)=x^tTx$. For an ordering $P$ of $K$, write $\sgnop_P(q)$ for the signature of $q$ at $P$; if $q$ is degenerate, this is the signature of the induced nondegenerate form on $V/\rad(q)$. We write $\totalSgn q=0$ to mean $\sgnop_P(q)=0$ for every ordering $P$. A field is formally real if and only if it admits an ordering. If $K$ has no ordering, the condition $\totalSgn q=0$ is vacuous.

\begin{definition}
\label{def:hyperbolic-systems}
Let $\cQ=\{q_i\}_{i\in I}$ be a system of quadratic forms on $V$.
\begin{enumerate}[label=\textup{(\roman*)}]
\item The system $\cQ$ is \emph{hyperbolic} if there are subspaces $U,W\subseteq V$ such that
\[
 V=U\oplus W,
 \qquad
 q_i|_{U\times U}=q_i|_{W\times W}=0
 \quad(i\in I).
\]
\item It is \emph{weakly hyperbolic} if $n\times\cQ$ is hyperbolic for some $n\geq1$.
\item A pair $\{q_1,q_2\}$ is \emph{nonsingular} if there is a field extension $L/K$ for which
\[
 \Span_L\{(q_1)_L,(q_2)_L\}
\]
contains a nondegenerate form. Otherwise it is \emph{singular}.
\end{enumerate}
\end{definition}

Here $n\times\cQ=\{n\times q_i\}_{i\in I}$ is the orthogonal sum of $n$ copies of the system on $V^{\oplus n}$. In this definition, $U$ and $W$ need not have the same dimension, and the forms may be degenerate. For a single nondegenerate form, hyperbolicity has its usual meaning.

For two systems with the same index set,
\[
 \cQ=\{q_i\}_{i\in I}\text{ on }V,
 \qquad
 \cQ'=\{q_i'\}_{i\in I}\text{ on }V',
\]
their orthogonal sum is
\[
 \cQ\perp\cQ'=\{q_i\oplus q_i'\}_{i\in I}.
\]

Following First \cite[Section~4]{First2020}, the algebra of adjoints of $\cQ$ is
\begin{equation}
\label{eq:adjoint-algebra}
 A(\cQ)=
 \left\{
 (\phi,\psi^{\op}):
 \psi^tq_i=q_i\phi\text{ for all }i\in I
 \right\}
 \subseteq\End(V)\times\End(V)^{\op}.
\end{equation}
The multiplication and canonical involution are
\[
 (\phi,\psi^{\op})(\phi',({\psi'})^{\op})
 =
 (\phi\phi',(\psi'\psi)^{\op}),
\]
\[
 (\phi,\psi^{\op})^\sigma=(\psi,\phi^{\op}).
\]
The involution-trace form is
\begin{equation}
\label{eq:trace-form-definition}
 q_{A(\cQ),\sigma}(a)
 =\Tr_{A(\cQ)/K}(a^\sigma a),
\end{equation}
where $\Tr_{A(\cQ)/K}(x)$ denotes the trace of left multiplication by $x$ on the finite-dimensional $K$-algebra $A(\cQ)$.

\begin{definition}
\label{def:adjoint-closure}
The \emph{adjoint closure} of $\cQ$ is
\begin{equation}
\label{eq:closure-definition}
 \Cl(\cQ)=
 \left\{
 q=q^t:\ \psi^tq=q\phi
 \text{ for every }(\phi,\psi^{\op})\in A(\cQ)
 \right\}.
\end{equation}
\end{definition}

By definition,
\[
 \Span_K\{q_i:i\in I\}\subseteq\Cl(\cQ).
\]
The inclusion can be strict; see \cref{prop:R-adjoint-closure}. We first record two elementary properties of the closure.

\begin{lemma}
\label{lem:hyperbolic-closure}
If $\cQ$ is hyperbolic, then every form in $\Cl(\cQ)$ is hyperbolic with respect to the same decomposition.
\end{lemma}

\begin{proof}
Choose $V=U\oplus W$ such that every member of $\cQ$ vanishes on $U$ and $W$. Let $e_U$ and $e_W$ be the corresponding projections. For every $q_i\in\cQ$ and $x,y\in V$,
\[
 q_i(e_Wx,y)=q_i(x,e_Uy),
\]
so
\[
 (e_U,e_W^{\op})\in A(\cQ).
\]
If $q\in\Cl(\cQ)$, then
\[
 q(e_Wx,y)=q(x,e_Uy).
\]
Taking $x,y\in U$ and then $x,y\in W$ gives $q|_{U\times U}=q|_{W\times W}=0$.
\end{proof}

\begin{lemma}[Amplification of the closure]
\label{lem:amplification}
For every system $\cQ$ and every $n\geq1$,
\[
 \Cl(n\times\cQ)
 =
 \{I_n\otimes q:q\in\Cl(\cQ)\}.
\]
\end{lemma}

\begin{proof}
Write the forms of $n\times\cQ$ as $I_n\otimes q_i$. For every $X\in\Mat_n(K)$,
\[
 \left(X\otimes I_V,(X^t\otimes I_V)^{\op}\right)
 \in A(n\times\cQ).
\]
Consequently, if $T\in\Cl(n\times\cQ)$, then
\[
 (X\otimes I_V)T=T(X\otimes I_V)
 \qquad\text{for all }X\in\Mat_n(K).
\]
The commutant of $\Mat_n(K)\otimes I_V$ is $I_n\otimes\End(V)$, so $T=I_n\otimes q$ for some symmetric form $q$ on $V$.

Let $(\phi,\psi^{\op})\in A(\cQ)$. Embedding this pair into one diagonal matrix block of $A(n\times\cQ)$ shows that $\psi^tq=q\phi$, hence $q\in\Cl(\cQ)$.

Conversely, suppose $q\in\Cl(\cQ)$ and let
\[
 (\Phi,\Psi^{\op})\in A(n\times\cQ),
 \qquad
 \Phi=(\phi_{ab}),\quad\Psi=(\psi_{ab}).
\]
The adjoint equations give
\[
 \psi_{ba}^tq_i=q_i\phi_{ab}
 \qquad\text{for all }a,b,i.
\]
Thus $(\phi_{ab},\psi_{ba}^{\op})\in A(\cQ)$ and
\[
 \psi_{ba}^tq=q\phi_{ab}.
\]
Blockwise, this is precisely
\[
 \Psi^t(I_n\otimes q)=(I_n\otimes q)\Phi.
\]
Therefore $I_n\otimes q\in\Cl(n\times\cQ)$.
\end{proof}

The preceding lemmas give the following consequence, noted by First in \cite[Introduction]{First2020}.

\begin{corollary}
\label{cor:closure-obstruction}
Let $K$ be formally real. If $\Cl(\cQ)$ contains a form with nonzero signature at some ordering of $K$, then $\cQ$ is not weakly hyperbolic.
\end{corollary}

\begin{proof}
Suppose $n\times\cQ$ is hyperbolic. If $q\in\Cl(\cQ)$, then $I_n\otimes q\in\Cl(n\times\cQ)$ by \cref{lem:amplification}; it is hyperbolic by \cref{lem:hyperbolic-closure}. Its signature is $n$ times the signature of $q$, so the latter must vanish at every ordering.
\end{proof}

We shall use the following two criteria of First \cite[Theorems~A and~B]{First2020}.

\begin{theorem}[First's general criterion]
\label{thm:first-A}
Let $\cQ$ be a finite system of quadratic forms over $K$. Then $\cQ$ is weakly hyperbolic if and only if
\[
 \totalSgn q_{A(\cQ),\sigma}=0.
\]
When these conditions hold, the smallest $n$ for which $n\times\cQ$ is hyperbolic is a power of $2$.
\end{theorem}

\begin{theorem}[First's closure criterion]
\label{thm:first-B}
Suppose that $K$ is a number field or a real closed field, and let $\cQ=\{q_1,q_2\}$ be a nonsingular pair. Then $\cQ$ is weakly hyperbolic if and only if
\[
 \totalSgn q=0
 \qquad\text{for every }q\in\Cl(\cQ).
\]
\end{theorem}

The first criterion applies to arbitrary finite systems. The second uses the signatures of forms in the adjoint closure and assumes that the pair is nonsingular.

To remove hyperbolic summands, we use the Hermitian-category interpretation of systems of forms \cite{BFFM2014,QSS1979}. For a fixed finite index set, First's equivalence \cite[Theorem~5.1]{First2020} identifies systems with unimodular Hermitian spaces in a twisted-arrow category and preserves orthogonal sums and hyperbolicity. In this category, idempotents split componentwise and endomorphism rings are finite-dimensional $K$-algebras. These rings are complete semilocal, with $2$ invertible. Proposition~5.12 of \cite{BFFM2014} therefore applies: a unimodular Hermitian space has zero Witt class if and only if it is hyperbolic.

\begin{lemma}[Removing a hyperbolic summand]
\label{lem:remove-hyperbolic}
Let $\cR$ and $\cS$ be systems of quadratic forms over $K$ with the same finite index set. If $\cS$ is hyperbolic, then
\[
 \cR\perp\cS\text{ is weakly hyperbolic}
 \quad\Longleftrightarrow\quad
 \cR\text{ is weakly hyperbolic}.
\]
\end{lemma}

\begin{proof}
If $n\times\cR$ is hyperbolic, then $n\times\cS$ is hyperbolic and so is their orthogonal sum.

Conversely, suppose $n\times(\cR\perp\cS)$ is hyperbolic. Write $[\cR]$ and $[\cS]$ for the Witt classes of the corresponding Hermitian spaces. First's equivalence gives
\[
 n[\cR]+n[\cS]=0.
\]
Since $[\cS]=0$, we have $n[\cR]=0$. By \cite[Proposition~5.12]{BFFM2014}, the Hermitian space corresponding to $n\times\cR$ is hyperbolic. Hence $n\times\cR$ is hyperbolic as a system.
\end{proof}

We now fix the singular blocks used in the Kronecker decomposition. For $\varepsilon\geq1$, set
\[
 E_\varepsilon=K^{\varepsilon+1},
 \qquad
 F_\varepsilon=K^\varepsilon,
\]
with bases $e_0,\ldots,e_\varepsilon$ and $f_0,\ldots,f_{\varepsilon-1}$. Let
\[
 L_{\varepsilon,0},L_{\varepsilon,1}:F_\varepsilon\longrightarrow E_\varepsilon
\]
be the inclusions
\[
 L_{\varepsilon,0}(f_j)=e_j,
 \qquad
 L_{\varepsilon,1}(f_j)=e_{j+1}
 \qquad(0\leq j\leq\varepsilon-1).
\]
The \emph{symmetric singular Kronecker block of minimal index $\varepsilon$} is the pair
\[
 M_\varepsilon
 =
 \{S_{\varepsilon,1},S_{\varepsilon,2}\}
\]
on $E_\varepsilon\oplus F_\varepsilon$, where
\begin{equation}
\label{eq:general-kronecker-block}
 S_{\varepsilon,1}
 =
 \begin{pmatrix}
 0&L_{\varepsilon,0}\\
 L_{\varepsilon,0}^t&0
 \end{pmatrix},
 \qquad
 S_{\varepsilon,2}
 =
 \begin{pmatrix}
 0&L_{\varepsilon,1}\\
 L_{\varepsilon,1}^t&0
 \end{pmatrix}.
\end{equation}
Its dimension is $2\varepsilon+1$. For $\varepsilon=0$, set $E_0=K$, $F_0=0$, and let $M_0$ be the one-dimensional zero pair.

Each $M_\varepsilon$ is hyperbolic for the decomposition
\[
 E_\varepsilon\oplus F_\varepsilon,
\]
and every member of its pencil is singular. The classical reduction theory for pairs of symmetric bilinear forms gives an orthogonal decomposition
\begin{equation}
\label{eq:kronecker-decomposition}
 \cQ
 \cong
 \cR\perp
 M_{\varepsilon_1}\perp\cdots\perp M_{\varepsilon_s},
\end{equation}
where $\cR$ is the regular, or nonsingular, part and the integers $\varepsilon_i\geq0$ are the singular minimal indices; see \cite{LeepSchueller1999,Waterhouse1976}. The regular part may have dimension zero, and $s=0$ exactly when the pair is nonsingular. The congruence class of each singular block is determined by its minimal index. We shall not need the classification of the regular part.

Pairs are indexed by $1,2$, so an invertible $K$-linear recombination is applied to both members in that order. Such a recombination preserves the adjoint algebra, adjoint closure, hyperbolicity, weak hyperbolicity, and singular minimal indices. Over an infinite field, every nonsingular pair has a nondegenerate member in its pencil \cite[Lemma~1.2]{First2020}; we may therefore normalize so that $R_1$ is nondegenerate whenever
\[
 J=R_1^{-1}R_2
\]
is used. The recombination is applied to every summand in \eqref{eq:kronecker-decomposition}. Congruences on the singular summands then restore the standard blocks in \eqref{eq:general-kronecker-block}, since their minimal indices are unchanged \cite{LeepSchueller1999,Waterhouse1976}.

\section{Construction of the counterexample}
\label{sec:construction}

We begin with First's four-dimensional example. On $K^4$, define
\begin{equation}
\label{eq:R-matrices}
 R_1=
 \begin{pmatrix}
 -1&0&0&0\\
 0&0&0&1\\
 0&0&1&0\\
 0&1&0&0
 \end{pmatrix},
 \qquad
 R_2=
 \begin{pmatrix}
 0&0&0&0\\
 0&0&0&0\\
 0&0&0&1\\
 0&0&1&0
 \end{pmatrix},
\end{equation}
and put
\[
 R_3=\operatorname{diag}(0,0,0,1),
 \qquad
 \cR=\{R_1,R_2\}.
\]
First considered this pair over $\bR$ in \cite[Example~2.2]{First2020}. The matrices define the same example over any field of characteristic different from $2$. We give the pencil and closure calculations over $K$, using the notation needed for the singular extension.

\begin{proposition}
\label{prop:R-pencil}
Every form in the pencil $\Span_K\{R_1,R_2\}$ is hyperbolic. Moreover,
\[
 \det(aR_1+bR_2)=a^4.
\]
In particular, $\cR$ is nonsingular.
\end{proposition}

\begin{proof}
For $a,b\in K$,
\[
 aR_1+bR_2=
 \begin{pmatrix}
 -a&0&0&0\\
 0&0&0&a\\
 0&0&a&b\\
 0&a&b&0
 \end{pmatrix},
\]
whose determinant is $a^4$. If $a\neq0$, the form is nondegenerate and the subspace
\[
 L=\Span_K\{e_2,e_1+e_3\}
\]
is totally isotropic of dimension $2$. A nondegenerate four-dimensional form with a two-dimensional totally isotropic subspace is hyperbolic.

If $a=0$, then $bR_2$ vanishes on both
\[
 U=\Span_K\{e_1,e_3\},
 \qquad
 W=\Span_K\{e_2,e_4\},
\]
and $K^4=U\oplus W$. Thus the degenerate form $bR_2$ is hyperbolic in the sense used for systems.
\end{proof}

We next compute the adjoint algebra and closure of $\cR$.

\begin{proposition}
\label{prop:R-adjoint-closure}
The elements of $A(\cR)$ are precisely the pairs $(\Phi_R,\Psi_R^{\op})$ with
\begin{equation}
\label{eq:R-adjoints}
 \Phi_R=
 \begin{pmatrix}
 x&0&0&-z\\
 -y&u&v&w\\
 0&0&u&v\\
 0&0&0&u
 \end{pmatrix},
 \qquad
 \Psi_R=
 \begin{pmatrix}
 x&0&0&y\\
 z&u&v&w\\
 0&0&u&v\\
 0&0&0&u
 \end{pmatrix},
\end{equation}
where $x,y,z,u,v,w\in K$. Furthermore,
\begin{equation}
\label{eq:R-closure}
 \Cl(\cR)
 =
 \Span_K\{R_1,R_2,R_3\}.
\end{equation}
\end{proposition}

\begin{proof}
Since $R_1$ is invertible, the first adjoint equation determines
\[
 \Psi_R^t=R_1\Phi_RR_1^{-1}.
\]
The second is then equivalent to $\Phi_RJ=J\Phi_R$, where
\[
 J=R_1^{-1}R_2=
 \begin{pmatrix}
 0&0&0&0\\
 0&0&1&0\\
 0&0&0&1\\
 0&0&0&0
 \end{pmatrix}.
\]
Solving this commutation relation gives the six-parameter matrix $\Phi_R$ in \eqref{eq:R-adjoints}; the first adjoint equation then gives $\Psi_R$. Substitution verifies the converse. We also have $J^3=0$ and $R_1J^2=R_3$.

Let $T=(t_{ij})$ be symmetric and suppose $T\in\Cl(\cR)$. Imposing
\[
 \Psi_R^tT=T\Phi_R
\]
for all values of $x,y,z,u,v,w$ gives
\[
 t_{12}=t_{13}=t_{14}=t_{22}=t_{23}=0,
 \qquad
 t_{24}=-t_{11},
 \qquad
 t_{33}=-t_{11},
\]
while $t_{34}$ and $t_{44}$ remain free. Thus
\[
 T=
 \begin{pmatrix}
 -\alpha&0&0&0\\
 0&0&0&\alpha\\
 0&0&\alpha&\beta\\
 0&\alpha&\beta&\gamma
 \end{pmatrix}
 =\alpha R_1+\beta R_2+\gamma R_3.
\]
Conversely, every linear combination of $R_1,R_2,R_3$ satisfies the closure equations, so \eqref{eq:R-closure} follows.
\end{proof}

\begin{corollary}
\label{cor:R-not-weak}
If $K$ is formally real, then $\cR$ is not weakly hyperbolic.
\end{corollary}

\begin{proof}
The form $R_3$ belongs to $\Cl(\cR)$ and has signature $1$ at every ordering of $K$. Apply \cref{cor:closure-obstruction}.
\end{proof}

For the singular summand, take the block $M_1$. On $K^3$, define
\begin{equation}
\label{eq:S-matrices}
 S_1=
 \begin{pmatrix}
 0&0&1\\
 0&0&0\\
 1&0&0
 \end{pmatrix},
 \qquad
 S_2=
 \begin{pmatrix}
 0&0&0\\
 0&0&1\\
 0&1&0
 \end{pmatrix},
\end{equation}
and put $\cS=\{S_1,S_2\}$.

\begin{proposition}
\label{prop:S-basic}
The system $\cS$ is hyperbolic, and every member of its pencil is singular.
\end{proposition}

\begin{proof}
Let $f_1,f_2,f_3$ be the standard basis of $K^3$, and set
\[
 U=\Span_K\{f_1,f_2\},
 \qquad
 W=\Span_K\{f_3\}.
\]
Both $S_1$ and $S_2$ vanish on $U$ and on $W$, and $K^3=U\oplus W$. Hence $\cS$ is hyperbolic.

For $a,b\in K$,
\[
 aS_1+bS_2=
 \begin{pmatrix}
 0&0&a\\
 0&0&b\\
 a&b&0
 \end{pmatrix},
\]
which has determinant zero.
\end{proof}

\begin{proposition}
\label{prop:S-adjoint-closure}
The elements of $A(\cS)$ are precisely the pairs $(\Phi_S,\Psi_S^{\op})$ with
\begin{equation}
\label{eq:S-adjoints}
 \Phi_S=
 \begin{pmatrix}
 p&0&q\\
 0&p&r\\
 0&0&s
 \end{pmatrix},
 \qquad
 \Psi_S=
 \begin{pmatrix}
 s&0&q\\
 0&s&r\\
 0&0&p
 \end{pmatrix},
\end{equation}
where $p,q,r,s\in K$. Moreover,
\begin{equation}
\label{eq:S-closure}
 \Cl(\cS)=\Span_K\{S_1,S_2\}.
\end{equation}
\end{proposition}

\begin{proof}
Solving
\[
 \Psi_S^tS_i=S_i\Phi_S
 \qquad(i=1,2)
\]
gives \eqref{eq:S-adjoints}. For a symmetric matrix $T=(t_{ij})$, the identities
\[
 \Psi_S^tT=T\Phi_S
\]
for all $p,q,r,s$ yield
\[
 t_{11}=t_{12}=t_{22}=t_{33}=0,
\]
with $t_{13}$ and $t_{23}$ arbitrary. Therefore
\[
 T=t_{13}S_1+t_{23}S_2,
\]
which proves \eqref{eq:S-closure}.
\end{proof}

The pencil can be written in the block form
\begin{equation}
\label{eq:Kronecker-block}
 aS_1+bS_2
 =
 \begin{pmatrix}
 0_{2\times2}&L(a,b)\\
 L(a,b)^t&0
 \end{pmatrix},
 \qquad
 L(a,b)=\begin{pmatrix}a\\b\end{pmatrix}.
\end{equation}
This is $M_1$ in the notation of \eqref{eq:general-kronecker-block}, of minimal index $1$ and dimension $3$. It is the smallest singular block of positive minimal index.

We can now form the counterexample. On $K^7=K^4\oplus K^3$, define
\begin{equation}
\label{eq:Q-definition}
 Q_1=R_1\oplus S_1,
 \qquad
 Q_2=R_2\oplus S_2,
 \qquad
 \cQ=\{Q_1,Q_2\}.
\end{equation}

\begin{proposition}
\label{prop:Q-singular-pencil}
The pair $\cQ$ is singular. Every form in its pencil is hyperbolic.
\end{proposition}

\begin{proof}
For every field extension $L/K$ and every $a,b\in L$, the block $aS_1+bS_2$ is singular. Hence
\[
 aQ_1+bQ_2=(aR_1+bR_2)\oplus(aS_1+bS_2)
\]
is singular. Thus $\cQ$ is a singular pair.

By \cref{prop:R-pencil}, the first summand is hyperbolic; by \cref{prop:S-basic}, the second is hyperbolic. Their orthogonal sum is hyperbolic.
\end{proof}

The closure of this orthogonal sum is smaller than the direct sum of the two closures. The following cross-adjoint supplies the additional constraints.

\begin{lemma}[A cross-adjoint]
\label{lem:cross-adjoint}
Let
\[
 U=
 \begin{pmatrix}
 0&0&0\\
 0&0&0\\
 0&0&0\\
 0&0&1
 \end{pmatrix}
 \in\Mat_{4\times3}(K),
 \qquad
 Y=
 \begin{pmatrix}
 0&1&0&0\\
 0&0&1&0\\
 0&0&0&0
 \end{pmatrix}
 \in\Mat_{3\times4}(K).
\]
Then
\[
 Y^tS_i=R_iU
 \qquad(i=1,2).
\]
Consequently, with
\[
 \Phi=
 \begin{pmatrix}0&U\\0&0\end{pmatrix},
 \qquad
 \Psi=
 \begin{pmatrix}0&0\\Y&0\end{pmatrix},
\]
we have $(\Phi,\Psi^{\op})\in A(\cQ)$.
\end{lemma}

\begin{proof}
Direct multiplication gives
\[
 Y^tS_1=R_1U=
 \begin{pmatrix}
 0&0&0\\
 0&0&1\\
 0&0&0\\
 0&0&0
 \end{pmatrix},
 \qquad
 Y^tS_2=R_2U=
 \begin{pmatrix}
 0&0&0\\
 0&0&0\\
 0&0&1\\
 0&0&0
 \end{pmatrix}.
\]
The block adjoint equation
\[
 \Psi^tQ_i=Q_i\Phi
\]
is exactly $Y^tS_i=R_iU$.
\end{proof}

\begin{theorem}[Closure calculation]
\label{thm:Q-closure}
The adjoint closure of $\cQ$ is exactly its pencil:
\[
 \Cl(\cQ)=\Span_K\{Q_1,Q_2\}.
\]
Consequently, every form in $\Cl(\cQ)$ is hyperbolic.
\end{theorem}

\begin{proof}
Let $P_R$ and $P_S$ be the projections of $K^4\oplus K^3$ onto the two direct summands. Since the $Q_i$ are block diagonal,
\[
 (P_R,P_R^{\op}),\ (P_S,P_S^{\op})\in A(\cQ).
\]
If $T\in\Cl(\cQ)$, the closure equations for these projections force $T$ to be block diagonal:
\[
 T=T_R\oplus T_S.
\]
Every adjoint of $\cR$ or $\cS$ extends by zero to an adjoint of $\cQ$. Hence
\[
 T_R\in\Cl(\cR),
 \qquad
 T_S\in\Cl(\cS).
\]
By \cref{prop:R-adjoint-closure,prop:S-adjoint-closure}, there are $a,b,c,d,e\in K$ such that
\[
 T_R=aR_1+bR_2+cR_3,
 \qquad
 T_S=dS_1+eS_2.
\]

Apply the closure equation to the cross-adjoint in \cref{lem:cross-adjoint}. It becomes
\[
 Y^tT_S=T_RU.
\]
Direct calculation gives
\[
 Y^tT_S-T_RU
 =
 \begin{pmatrix}
 0&0&0\\
 0&0&d-a\\
 0&0&e-b\\
 0&0&-c
 \end{pmatrix}.
\]
Therefore
\[
 d=a,
 \qquad
 e=b,
 \qquad
 c=0.
\]
It follows that
\[
 T=a(R_1\oplus S_1)+b(R_2\oplus S_2)
 =aQ_1+bQ_2.
\]
Thus $\Cl(\cQ)\subseteq\Span_K\{Q_1,Q_2\}$. The reverse inclusion holds for the closure of every system, proving the equality.

Every member of the pencil is hyperbolic by \cref{prop:Q-singular-pencil}.
\end{proof}

\begin{proposition}[Failure of weak hyperbolicity]
\label{prop:conceptual-obstruction}
If $K$ is formally real, then $\cQ$ is not weakly hyperbolic.
\end{proposition}

\begin{proof}
The singular system $\cS$ is hyperbolic. By \cref{lem:remove-hyperbolic},
\[
 \cR\perp\cS\text{ is weakly hyperbolic}
 \quad\Longleftrightarrow\quad
 \cR\text{ is weakly hyperbolic}.
\]
The latter condition fails by \cref{cor:R-not-weak}.
\end{proof}

This completes the proof of \cref{thm:main-intro}.

\section{Closure formulas and the dimension bound}
\label{sec:visibility}

Let
\[
 \cR=\{R_1,R_2\}\text{ on }V_R,
 \qquad
 \cS=\{S_1,S_2\}\text{ on }V_S,
\]
and set
\[
 \cQ=\cR\perp\cS
 =
 \{R_1\oplus S_1,R_2\oplus S_2\}.
\]
Define the cross-adjoint spaces
\begin{align}
 \cH_{RS}
 &=
 \left\{
 (U,Y):
 U\in\Hom(V_S,V_R),\
 Y\in\Hom(V_R,V_S),\
 Y^tS_i=R_iU\ (i=1,2)
 \right\},
 \label{eq:HRS}\\
 \cH_{SR}
 &=
 \left\{
 (W,X):
 W\in\Hom(V_R,V_S),\
 X\in\Hom(V_S,V_R),\
 X^tR_i=S_iW\ (i=1,2)
 \right\}.
 \label{eq:HSR}
\end{align}

\begin{proposition}[Closure of an orthogonal sum]
\label{prop:cross-closure}
A symmetric form $T$ belongs to $\Cl(\cQ)$ if and only if it has the form
\[
 T=T_R\oplus T_S,
 \qquad
 T_R\in\Cl(\cR),
 \quad
 T_S\in\Cl(\cS),
\]
and satisfies
\begin{align}
 Y^tT_S&=T_RU
 &&\text{for every }(U,Y)\in\cH_{RS},
 \label{eq:cross-closure-one}\\
 X^tT_R&=T_SW
 &&\text{for every }(W,X)\in\cH_{SR}.
 \label{eq:cross-closure-two}
\end{align}
\end{proposition}

\begin{proof}
The block projections belong to $A(\cQ)$, so every form in the closure is block diagonal. Every adjoint of either diagonal summand extends by zero to an adjoint of $\cQ$, and hence the two diagonal blocks of a member of $\Cl(\cQ)$ lie in $\Cl(\cR)$ and $\Cl(\cS)$.

Write an arbitrary adjoint of $\cQ$ in block form as
\[
 \Phi=\begin{pmatrix}A&U\\W&C\end{pmatrix},
 \qquad
 \Psi=\begin{pmatrix}B&X\\Y&D\end{pmatrix},
\]
where $(A,B^{\op})\in A(\cR)$, $(C,D^{\op})\in A(\cS)$, $(U,Y)\in\cH_{RS}$, and $(W,X)\in\cH_{SR}$. For $T=T_R\oplus T_S$, the off-diagonal blocks of
\[
 \Psi^tT=T\Phi
\]
are precisely \eqref{eq:cross-closure-one} and \eqref{eq:cross-closure-two}; the diagonal blocks are the closure equations for $T_R$ and $T_S$.
\end{proof}

\begin{definition}
Assume that $\cR$ is nonsingular. The \emph{visible regular closure} of $\cR$ in $\cQ=\cR\perp\cS$ is
\[
 \Vis_{\cR}(\cQ)
 =
 \left\{
 T_R\in\Cl(\cR):
 \text{there is }T_S\in\Cl(\cS)
 \text{ with }T_R\oplus T_S\in\Cl(\cQ)
 \right\}.
\]
Equivalently, it is the image of the restriction map
\[
 \rho_R:\Cl(\cQ)\longrightarrow\Cl(\cR),
 \qquad T_R\oplus T_S\longmapsto T_R.
\]
\end{definition}

We first determine the closure of an individual singular block.

\begin{lemma}
\label{lem:closure-kronecker-block}
For every $\varepsilon\geq1$,
\[
 \Cl(M_\varepsilon)
 =
 \Span_K\{S_{\varepsilon,1},S_{\varepsilon,2}\}.
\]
For $\varepsilon=0$, one has $\Cl(M_0)=0$.
\end{lemma}

\begin{proof}
Fix $\varepsilon\geq1$ and write $E=E_\varepsilon$, $F=F_\varepsilon$. Since $M_\varepsilon$ is hyperbolic with respect to $E\oplus F$, \cref{lem:hyperbolic-closure} implies that every $T\in\Cl(M_\varepsilon)$ has the form
\[
 T=
 \begin{pmatrix}
 0&C\\
 C^t&0
 \end{pmatrix},
 \qquad
 C\in\Mat_{(\varepsilon+1)\times\varepsilon}(K).
\]
Index the rows of $C$ by $0,\ldots,\varepsilon$ and its columns by $0,\ldots,\varepsilon-1$.

For $0\leq r\leq2\varepsilon-1$, let $H_r\in\Mat_{(\varepsilon+1)\times\varepsilon}(K)$ be the Hankel matrix
\[
 (H_r)_{ij}
 =
 \begin{cases}
 1,&i+j=r,\\
 0,&i+j\neq r.
 \end{cases}
\]
Put
\[
 \Phi_r=\Psi_r=
 \begin{pmatrix}
 0&H_r\\
 0&0
 \end{pmatrix}.
\]
Because $H_r^tL_{\varepsilon,0}$ and $H_r^tL_{\varepsilon,1}$ are symmetric,
\[
 (\Phi_r,\Psi_r^{\op})\in A(M_\varepsilon).
\]
The closure equation therefore gives
\begin{equation}
\label{eq:hankel-closure}
 H_r^tC=C^tH_r
 \qquad(0\leq r\leq2\varepsilon-1).
\end{equation}

Write $C=(c_{ij})$ and interpret $c_{ij}$ as zero when the row index is outside $\{0,\ldots,\varepsilon\}$. If $\varepsilon=1$, then $C$ is a $2\times1$ matrix and is already of the form $aL_{1,0}+bL_{1,1}$. We may therefore assume $\varepsilon\geq2$. For $0\leq j<\ell\leq\varepsilon-1$, comparison of the $(j,\ell)$-entry in \eqref{eq:hankel-closure} gives
\begin{equation}
\label{eq:hankel-index}
 c_{r-j,\ell}=c_{r-\ell,j}.
\end{equation}
Taking $j=0$ and $r=i<\ell$ shows that $c_{i\ell}=0$ whenever $i<\ell$. If $i>j+1$, then $j\leq\varepsilon-2$. Choose $\ell=\varepsilon-1$ and $r=i+\ell$, so that $j<\ell$ and $r\leq2\varepsilon-1$; the left side of \eqref{eq:hankel-index} has row index
\[
 i+\ell-j> \varepsilon,
\]
and hence $c_{ij}=0$. Thus
\[
 c_{ij}=0
 \qquad\text{unless}\qquad
 i=j\text{ or }i=j+1.
\]
Taking $r=j+\ell$ in \eqref{eq:hankel-index} shows that the diagonal entries $c_{jj}$ are equal. Taking $r=j+\ell+1$ shows that the subdiagonal entries $c_{j+1,j}$ are equal. Hence
\[
 C=aL_{\varepsilon,0}+bL_{\varepsilon,1}
\]
for some $a,b\in K$, and therefore
\[
 T=aS_{\varepsilon,1}+bS_{\varepsilon,2}.
\]
The reverse inclusion holds for every system.

For $\varepsilon=0$, $M_0$ is the zero pair on a one-dimensional space. Its adjoint algebra contains $(1,0^{\op})$, whose closure equation forces the form to be zero.
\end{proof}

The cross-adjoints with a regular summand determine its restriction in the total closure.

\begin{theorem}[Closure collapse]
\label{thm:closure-collapse}
Let $K$ be a field of characteristic different from $2$, and let
\[
 \cR=\{R_1,R_2\}
\]
be a pair with $R_1$ nondegenerate. For $\varepsilon\geq1$, set
\[
 \cQ=\cR\perp M_\varepsilon.
\]
Then
\[
 \Cl(\cQ)
 =
 \Span_K
 \{R_1\perp S_{\varepsilon,1},
   R_2\perp S_{\varepsilon,2}\}.
\]
Consequently,
\[
 \Vis_{\cR}(\cQ)
 =
 \Span_K\{R_1,R_2\}.
\]
\end{theorem}

\begin{proof}
Set
\[
 J=R_1^{-1}R_2.
\]
Let
\[
 T=T_R\oplus T_S\in\Cl(\cQ).
\]
By \cref{prop:cross-closure,lem:closure-kronecker-block},
\[
 T_R\in\Cl(\cR),
 \qquad
 T_S=aS_{\varepsilon,1}+bS_{\varepsilon,2}
\]
for some $a,b\in K$.

For $v\in V_R$, define
\[
 U_v=
 \left(
 0_{\dim_K(V_R)\times(\varepsilon+1)}
 \ \middle|\
 v,\ Jv,\ldots,J^{\varepsilon-1}v
 \right)
 \in\Hom(E_\varepsilon\oplus F_\varepsilon,V_R)
\]
and define $Y_v\in\Hom(V_R,E_\varepsilon\oplus F_\varepsilon)$ by
\[
 Y_v^t=
 \left(
 R_1v,\ R_1Jv,\ldots,R_1J^\varepsilon v
 \ \middle|\
 0_{\dim_K(V_R)\times\varepsilon}
 \right).
\]
Using \eqref{eq:general-kronecker-block},
\[
 Y_v^tS_{\varepsilon,1}=R_1U_v,
 \qquad
 Y_v^tS_{\varepsilon,2}=R_2U_v.
\]
Hence $(U_v,Y_v)\in\cH_{RS}$. Applying
\eqref{eq:cross-closure-one} gives
\[
 (T_R-aR_1-bR_2)U_v=0.
\]
The first column of the $F_\varepsilon$-block of $U_v$ is $v$, so
\[
 (T_R-aR_1-bR_2)v=0.
\]
Since $v$ is arbitrary,
\[
 T_R=aR_1+bR_2.
\]
Therefore
\[
 T
 =
 a(R_1\perp S_{\varepsilon,1})
 +
 b(R_2\perp S_{\varepsilon,2}).
\]
This proves one inclusion; the reverse is the general inclusion of the pencil in its adjoint closure.
\end{proof}

Minimal index zero behaves differently.

\begin{lemma}
\label{lem:zero-block}
Let $Z_m$ denote the zero pair on $K^m$, and let $\cR$ be a pair whose pencil contains a nondegenerate form. Then
\[
 \Cl(\cR\perp Z_m)
 =
 \{T_R\perp0:T_R\in\Cl(\cR)\}.
\]
In particular,
\[
 \Vis_{\cR}(\cR\perp Z_m)=\Cl(\cR).
\]
\end{lemma}

\begin{proof}
After an invertible linear recombination of the pair, assume that $R_1$ is nondegenerate. In the notation of \cref{prop:cross-closure}, the cross-adjoint equations with $S_1=S_2=0$ give
\[
 R_1U=0,\qquad X^tR_1=0,
\]
so $U=X=0$, while $Y$ and $W$ are arbitrary. The block projections force every closure form to be $T_R\oplus T_0$, with $T_R\in\Cl(\cR)$. The adjoint $(I_{K^m},0^{\op})$ of the zero pair, extended by zero on $V_R$, then forces $T_0=0$. Conversely, for every $T_R\in\Cl(\cR)$, both cross-closure equations are satisfied by $T_R\oplus0$, because $U=X=0$. This proves the equality.
\end{proof}

\begin{theorem}[Visibility dichotomy]
\label{thm:visibility-dichotomy}
Let $K$ be an infinite field of characteristic different from $2$, and let
\[
 \cQ
 \cong
 \cR\perp
 M_{\varepsilon_1}\perp\cdots\perp M_{\varepsilon_s}
\]
be a Kronecker decomposition with $s\geq1$ and $\cR$ nonsingular. Then
\[
 \Vis_{\cR}(\cQ)
 =
 \begin{cases}
 \Cl(\cR),
 &
 \varepsilon_1=\cdots=\varepsilon_s=0,\\[2mm]
 \Span_K\{R_1,R_2\},
 &
 \max_i\varepsilon_i\geq1.
 \end{cases}
\]
\end{theorem}

\begin{proof}
Apply an invertible linear recombination to the whole pair so that $R_1$ is nondegenerate, and restore the standard singular blocks by congruence, as explained after \eqref{eq:kronecker-decomposition}. The required equality is unchanged by these operations.

If all minimal indices are zero, the singular part is a zero pair and the assertion is \cref{lem:zero-block}.

Suppose that $\varepsilon_j\geq1$ for some $j$. The block projections belong to the adjoint algebra. Moreover, every adjoint of the subsystem
\[
 \cR\perp M_{\varepsilon_j}
\]
extends by zero on the remaining Kronecker summands to an adjoint of $\cQ$. Hence the restriction of any $T\in\Cl(\cQ)$ to this subsystem belongs to
\[
 \Cl(\cR\perp M_{\varepsilon_j}).
\]
By \cref{thm:closure-collapse}, its regular component lies in
\[
 \Span_K\{R_1,R_2\}.
\]
This proves one inclusion. The reverse follows because the pencil of $\cQ$ is contained in $\Cl(\cQ)$.
\end{proof}

When the regular pencil is two-dimensional, its coefficients determine the singular components as well.

\begin{theorem}[Complete closure dichotomy]
\label{thm:full-closure-dichotomy}
In the setting of \cref{thm:visibility-dichotomy}, assume in addition that
\[
 \dim_K\Span_K\{R_1,R_2\}=2.
\]
Write $\cQ=\{Q_1,Q_2\}$. Then
\[
 \Cl(\cQ)=
 \begin{cases}
 \{T_R\perp0:T_R\in\Cl(\cR)\},
 &\varepsilon_1=\cdots=\varepsilon_s=0,\\[2mm]
 \Span_K\{Q_1,Q_2\},
 &\max_i\varepsilon_i\geq1.
 \end{cases}
\]
\end{theorem}

\begin{proof}
Make the same simultaneous normalization as in the proof of \cref{thm:visibility-dichotomy}, so that $R_1$ is nondegenerate and the singular blocks are standard. If all minimal indices are zero, the singular part is a zero pair, so the first formula is \cref{lem:zero-block}.

Assume now that $\varepsilon_j\geq1$ for at least one $j$, and let $T\in\Cl(\cQ)$. The block projections onto the regular summand and the individual Kronecker summands belong to $A(\cQ)$. Hence
\[
 T=T_R\perp T_1\perp\cdots\perp T_s,
 \qquad
 T_R\in\Cl(\cR),\quad
 T_i\in\Cl(M_{\varepsilon_i}).
\]
By \cref{thm:visibility-dichotomy},
\[
 T_R=aR_1+bR_2
\]
for unique $a,b\in K$, because $R_1$ and $R_2$ are linearly independent.

If $\varepsilon_i=0$, then $T_i=0$ by \cref{lem:closure-kronecker-block}. If $\varepsilon_i\geq1$, every adjoint of the subsystem
\[
 \cR\perp M_{\varepsilon_i}
\]
extends by zero on the remaining Kronecker summands. Therefore
\[
 T_R\perp T_i\in\Cl(\cR\perp M_{\varepsilon_i}).
\]
By \cref{thm:closure-collapse}, there are $a_i,b_i\in K$ such that
\[
 T_R=a_iR_1+b_iR_2,
 \qquad
 T_i=a_iS_{\varepsilon_i,1}+b_iS_{\varepsilon_i,2}.
\]
The linear independence of $R_1,R_2$ gives $a_i=a$ and $b_i=b$. Hence
\[
 T=aQ_1+bQ_2.
\]
Thus $\Cl(\cQ)\subseteq\Span_K\{Q_1,Q_2\}$, and the reverse inclusion holds for every system.
\end{proof}

The closure formula gives the following extension of the construction in Section~\ref{sec:construction}.

\begin{corollary}[Singular extensions]
\label{cor:general-counterexamples}
Let $K$ be formally real and let $\cR=\{R_1,R_2\}$ be a nonsingular pair such that
\[
 \totalSgn(aR_1+bR_2)=0
 \qquad\text{for all }a,b\in K,
\]
while $\cR$ is not weakly hyperbolic. Then for every $\varepsilon\geq1$,
\[
 \cQ_\varepsilon=\{Q_{\varepsilon,1},Q_{\varepsilon,2}\}
 =\cR\perp M_\varepsilon,
 \qquad Q_{\varepsilon,i}=R_i\perp S_{\varepsilon,i},
\]
is singular and not weakly hyperbolic, whereas
\[
 \Cl(\cQ_\varepsilon)
 =
 \Span_K\{Q_{\varepsilon,1},Q_{\varepsilon,2}\}
\]
and every form in this closure has total signature zero.
\end{corollary}

\begin{proof}
A formally real field is infinite. Normalize the whole pair $\cQ_\varepsilon$ so that its first regular member is nondegenerate, and restore the standard singular block as explained after \eqref{eq:kronecker-decomposition}. Then \cref{thm:closure-collapse} applies, and its conclusion transfers back to the original pair. Every pencil member is singular because of its $M_\varepsilon$ summand. The latter is hyperbolic and has zero signature, so every form in $\Cl(\cQ_\varepsilon)$ has total signature zero.

The system $M_\varepsilon$ is hyperbolic. By \cref{lem:remove-hyperbolic},
\[
 \cQ_\varepsilon\text{ is weakly hyperbolic}
 \quad\Longleftrightarrow\quad
 \cR\text{ is weakly hyperbolic},
\]
which is false by hypothesis.
\end{proof}

Applied to the pair in \eqref{eq:R-matrices}, this gives counterexamples of dimension $2\varepsilon+5$ for every $\varepsilon\geq1$. The case $\varepsilon=1$ is \cref{thm:main-intro}.

It remains to prove the dimension bound in \cref{thm:optimal-intro}. We begin with nonsingular pairs in dimension two.

\begin{lemma}
\label{lem:dimension-two-closure}
Let $K$ be an infinite field of characteristic different from $2$, and let
\[
 \cR=\{R_1,R_2\}
\]
be a nonsingular pair on a two-dimensional vector space. Then
\[
 \Cl(\cR)=\Span_K\{R_1,R_2\}.
\]
\end{lemma}

\begin{proof}
After an invertible linear recombination, assume that $R_1$ is nondegenerate and put
\[
 J=R_1^{-1}R_2.
\]
The first adjoint equation determines $\psi$ from $\phi$; the second is equivalent to
\[
 \phi J=J\phi.
\]
Thus the first components of the adjoint algebra form the centralizer
\[
 \Cent_{\Mat_2(K)}(J).
\]

Let $T\in\Cl(\cR)$ and put
\[
 C=R_1^{-1}T.
\]
Since
\[
 \psi^t=R_1\phi R_1^{-1},
\]
the closure equation $\psi^tT=T\phi$ is equivalent to
\[
 \phi C=C\phi
 \qquad
 \text{for every }\phi\in\Cent_{\Mat_2(K)}(J).
\]
Hence
\[
 C\in
 \Cent_{\Mat_2(K)}
 \bigl(\Cent_{\Mat_2(K)}(J)\bigr).
\]

If $J$ is scalar, its centralizer is $\Mat_2(K)$ and its bicommutant is the scalar algebra. If $J$ is not scalar, its minimal polynomial has degree $2$, so
\[
 \Cent_{\Mat_2(K)}(J)=K[J],
\]
and the bicommutant is again $K[J]$. Thus in either case
\[
 T\in\Span_K\{R_1,R_2\}.
\] The reverse inclusion follows from the definition of the closure.
\end{proof}

\begin{proposition}
\label{prop:minimal-regular-obstruction}
Let $K$ be a number field or a real closed field, and let $\cR$ be a nonsingular pair on a vector space $V_R$. Suppose that every form in the pencil of $\cR$ has total signature zero, but $\cR$ is not weakly hyperbolic. Then
\[
 \dim_K V_R\geq4.
\]
For formally real number fields and real closed fields, the bound is sharp.
\end{proposition}

\begin{proof}
If $K$ has no ordering, every total signature is vacuously zero, including that of the involution-trace form of $A(\cR)$. By \cref{thm:first-A}, $\cR$ would then be weakly hyperbolic. Hence $K$ must admit an ordering.

Dimension $0$ is impossible because the zero system is already hyperbolic. Since $K$ is infinite, the pencil contains a nondegenerate form. In odd dimension such a form has odd signature at every ordering, excluding dimensions $1$ and $3$.

If $\dim_K V_R=2$, \cref{lem:dimension-two-closure} gives
\[
 \Cl(\cR)=\Span_K\{R_1,R_2\}.
\]
Thus every form in $\Cl(\cR)$ has total signature zero. First's closure criterion, \cref{thm:first-B}, then implies that $\cR$ is weakly hyperbolic, a contradiction.

For every formally real number field and every real closed field, the four-dimensional pair in \eqref{eq:R-matrices} shows that the bound is sharp.
\end{proof}

\begin{proof}[Proof of \cref{thm:optimal-intro}]
If $K$ has no ordering, First's general criterion, \cref{thm:first-A}, implies that every finite system is weakly hyperbolic, so the hypotheses are impossible. We may therefore assume that $K$ admits an ordering.

Write the Kronecker decomposition as
\[
 \cQ
 \cong
 \cR\perp
 M_{\varepsilon_1}\perp\cdots\perp M_{\varepsilon_s}.
\]
The singular part is hyperbolic, so repeated application of
\cref{lem:remove-hyperbolic} gives
\[
 \cQ\text{ is weakly hyperbolic}
 \quad\Longleftrightarrow\quad
 \cR\text{ is weakly hyperbolic}.
\]
Hence $\cR$ is not weakly hyperbolic; in particular, the regular part is nonzero.

If every $\varepsilon_i$ is zero, the singular part is the zero pair $Z_s$. By \cref{lem:zero-block},
\[
 \Cl(\cQ)=\{T_R\perp0:T_R\in\Cl(\cR)\}.
\]
Since $\sgnop_P(T_R\perp0)=\sgnop_P(T_R)$ at every ordering $P$, the hypothesis on $\Cl(\cQ)$ forces zero total signature on $\Cl(\cR)$. First's closure criterion would make $\cR$ weakly hyperbolic, a contradiction. Therefore
\[
 \varepsilon_i\geq1
\]
for at least one $i$, and the singular part has dimension at least $3$.

Because the pencil of $\cQ$ lies in $\Cl(\cQ)$, all its members have zero total signature. The singular Kronecker pencils are hyperbolic, so every form in the pencil of $\cR$ also has zero total signature. Since $\cR$ is not weakly hyperbolic, \cref{prop:minimal-regular-obstruction} gives
\[
 \dim_K V_R\geq4.
\]
Consequently,
\[
 \dim_K V
 \geq4+3=7.
\]
For every formally real number field and every real closed field, the example of \cref{thm:main-intro} has dimension $4+3=7$, so the bound is sharp.
\end{proof}

\begin{remark}
The proof uses First's nonsingular closure criterion only for zero singular blocks and for the two-dimensional regular part. The same dimension bound therefore holds over fields satisfying condition \emph{(E)} of \cite[Section~6]{First2020}, by \cite[Theorem~6.1]{First2020}. If such a field has no ordering, the statement is vacuous by \cref{thm:first-A}; otherwise it is infinite, as required by the normalization argument.
\end{remark}

\section{The adjoint algebra and trace obstruction}
\label{sec:adjoint-obstructions}

We return to the seven-dimensional pair $\cQ$ of \eqref{eq:Q-definition}. Write a pair of $7\times7$ matrices in $4+3$ block form as
\[
 \Phi=
 \begin{pmatrix}A&U\\W&C\end{pmatrix},
 \qquad
 \Psi=
 \begin{pmatrix}B&X\\Y&D\end{pmatrix}.
\]
The adjoint equations for $Q_i=R_i\oplus S_i$ split into
\begin{equation}
\label{eq:block-adjoint-equations}
 \begin{aligned}
 B^tR_i&=R_iA,&
 D^tS_i&=S_iC,\\
 Y^tS_i&=R_iU,&
 X^tR_i&=S_iW
 \end{aligned}
 \qquad(i=1,2).
\end{equation}
The first two equations concern the diagonal summands; the last two describe the cross-adjoints.

\begin{proposition}[Explicit adjoint algebra]
\label{prop:full-adjoint}
The algebra $A(\cQ)$ has dimension $18$. Its elements are precisely the pairs $(\Phi,\Psi^{\op})$ with the following blocks.

The regular diagonal blocks are
\[
 A=
 \begin{pmatrix}
 x&0&0&-z\\
 -y&u&v&w\\
 0&0&u&v\\
 0&0&0&u
 \end{pmatrix},
 \qquad
 B=
 \begin{pmatrix}
 x&0&0&y\\
 z&u&v&w\\
 0&0&u&v\\
 0&0&0&u
 \end{pmatrix}.
\]
The singular diagonal blocks are
\[
 C=
 \begin{pmatrix}
 p&0&q\\
 0&p&r\\
 0&0&s
 \end{pmatrix},
 \qquad
 D=
 \begin{pmatrix}
 s&0&q\\
 0&s&r\\
 0&0&p
 \end{pmatrix}.
\]
The first cross-adjoint space is
\[
 U=
 \begin{pmatrix}
 0&0&-a_0\\
 0&0&a_1\\
 0&0&a_3\\
 0&0&a_2
 \end{pmatrix},
 \qquad
 Y=
 \begin{pmatrix}
 a_0&a_2&a_3&a_1\\
 0&0&a_2&a_3\\
 0&0&0&0
 \end{pmatrix},
\]
and the opposite cross-adjoint space is
\[
 W=
 \begin{pmatrix}
 -b_0&b_3&b_2&b_1\\
 0&0&b_3&b_2\\
 0&0&0&0
 \end{pmatrix},
 \qquad
 X=
 \begin{pmatrix}
 0&0&b_0\\
 0&0&b_1\\
 0&0&b_2\\
 0&0&b_3
 \end{pmatrix}.
\]
All eighteen parameters
\[
 x,y,z,u,v,w,p,q,r,s,a_0,a_1,a_2,a_3,b_0,b_1,b_2,b_3
\]
are arbitrary elements of $K$.
\end{proposition}

\begin{proof}
The four systems in \eqref{eq:block-adjoint-equations} involve disjoint block variables. The two diagonal systems give the six- and four-parameter solutions from \cref{prop:R-adjoint-closure,prop:S-adjoint-closure}.

For the first cross system, write the columns of $U$ as $u_1,u_2,h$ and those of $Y^t$ as $y_1,y_2,y_3$. The equations $Y^tS_i=R_iU$ become
\[
 \begin{aligned}
 R_1u_1&=y_3,& R_1u_2&=0,& R_1h&=y_1,\\
 R_2u_1&=0,& R_2u_2&=y_3,& R_2h&=y_2.
 \end{aligned}
\]
Since $R_1$ is invertible, $u_2=0$, then $y_3=0$, and then $u_1=0$. Thus
\[
 U=(0,0,h),\qquad Y^t=(R_1h,R_2h,0),
\]
with $h\in K^4$ arbitrary. Taking $h=(-a_0,a_1,a_3,a_2)^t$ gives the displayed matrices. Transposing the opposite cross system gives $R_iX=W^tS_i$, so the same calculation yields
\[
 X=(0,0,k),\qquad W^t=(R_1k,R_2k,0).
\]
Taking $k=(b_0,b_1,b_2,b_3)^t$ gives the stated parametrization. Conversely, these formulas satisfy all four adjoint equations. The independent dimensions are therefore $6$, $4$, $4$, and $4$, giving
\[
 \dim_K A(\cQ)=6+4+4+4=18.
\]
\end{proof}

Let
\[
 \mathcal B=
 (E_x,E_y,E_z,E_u,E_v,E_w,
 E_p,E_q,E_r,E_s,
 E_{a_0},\ldots,E_{a_3},E_{b_0},\ldots,E_{b_3})
\]
be the basis obtained by setting one parameter in \cref{prop:full-adjoint} equal to $1$ and all others equal to $0$.

\begin{proposition}[Radical and semisimple quotient]
\label{prop:radical-structure}
Let $\mathfrak r\subseteq A(\cQ)$ be the subspace defined by
\[
 x=u=p=s=0.
\]
Then $\mathfrak r$ is a $\sigma$-stable nilpotent ideal, $\mathfrak r^5=0$, and
\[
 \mathfrak r=\Jac A(\cQ),
 \qquad
 A(\cQ)/\mathfrak r\cong K^4.
\]
Under the quotient map
\[
 \pi:A(\cQ)\longrightarrow K^4,
 \qquad
 a\longmapsto(x,u,p,s),
\]
multiplication is coordinatewise and the induced involution is
\[
 (x,u,p,s)^\sigma=(x,u,s,p).
\]
\end{proposition}

\begin{proof}
Multiplying the block matrices in \cref{prop:full-adjoint}, with the reversed order in the second component of the adjoint algebra, gives
\[
 \pi(ab)=
 (x_ax_b,u_au_b,p_ap_b,s_as_b)
 =\pi(a)\pi(b).
\]
Thus $\pi$ is a surjective algebra homomorphism with kernel $\mathfrak r$, and
\[
 A(\cQ)/\mathfrak r\cong K^4.
\]
The formulas show that $\sigma$ fixes $x$ and $u$, interchanges $p$ and $s$, and preserves the subspace $x=u=p=s=0$. Hence $\mathfrak r$ is $\sigma$-stable.

For nilpotence, let $e_1,\ldots,e_7$ be the standard basis of $K^7$ and consider the flag
\[
 \begin{aligned}
 F_0&=K^7,\\
 F_1&=\Span_K\{e_1,e_2,e_3,e_4,e_5,e_6\},\\
 F_2&=\Span_K\{e_1,e_2,e_3,e_5,e_6\},\\
 F_3&=\Span_K\{e_2,e_5,e_6\},\\
 F_4&=\Span_K\{e_5\},
 \qquad F_5=0.
 \end{aligned}
\]
If $j=(\Phi_j,\Psi_j^{\op})\in \mathfrak r$, the matrices in \cref{prop:full-adjoint} give
\[
 \Phi_j(F_i)\subseteq F_{i+1},
 \qquad
 \Psi_j(F_i)\subseteq F_{i+1}
 \qquad(0\leq i\leq4).
\]
Hence every product of five first components from $\mathfrak r$ is zero, and likewise for five second components in either order. Since multiplication in $A(\cQ)$ is
\[
 (\Phi,\Psi^{\op})(\Phi',({\Psi'})^{\op})
 =
 (\Phi\Phi',(\Psi'\Psi)^{\op}),
\]
we obtain $\mathfrak r^5=0$.

Thus $\mathfrak r\subseteq\Jac A(\cQ)$. Since $A(\cQ)/\mathfrak r\cong K^4$ is semisimple, the reverse inclusion also holds, and $\mathfrak r=\Jac A(\cQ)$.
\end{proof}

\begin{proposition}[Left regular trace and involution-trace form]
\label{prop:trace-form}
For an element $a\in A(\cQ)$ with the parameters of \cref{prop:full-adjoint},
\begin{equation}
\label{eq:regular-trace}
 \Tr_{A(\cQ)/K}(a)=3x+7u+7p+s.
\end{equation}
Consequently,
\[
 q_{A(\cQ),\sigma}(a)=3x^2+7u^2+8ps,
\]
and, as a quadratic form on $A(\cQ)$,
\begin{equation}
\label{eq:trace-form-explicit}
 q_{A(\cQ),\sigma}
 \cong
 \langle3,7\rangle
 \perp
 \begin{pmatrix}0&4\\4&0\end{pmatrix}
 \perp
 \langle0\rangle^{14}.
\end{equation}
\end{proposition}

\begin{proof}
Let $L_a$ denote left multiplication by $a$ on the eighteen-dimensional algebra $A(\cQ)$. In the basis $\mathcal B$, the diagonal coefficients of $L_a$ are as follows:
\[
\begin{array}{lc}
\hline
\text{Basis vectors}&\text{Diagonal coefficient of }L_a\\
\hline
E_x,E_z,E_{a_0}&x\\
E_y,E_u,E_v,E_w,E_{a_1},E_{a_2},E_{a_3}&u\\
E_p,E_q,E_r,E_{b_0},E_{b_1},E_{b_2},E_{b_3}&p\\
E_s&s\\
\hline
\end{array}
\]
The entries follow by multiplying the matrices in \cref{prop:full-adjoint}. Their sum gives \eqref{eq:regular-trace}.

By \cref{prop:radical-structure},
\[
 \pi(a^\sigma a)
 =(x,u,s,p)(x,u,p,s)
 =(x^2,u^2,sp,ps).
\]
Applying \eqref{eq:regular-trace} to $a^\sigma a$ gives
\[
 q_{A(\cQ),\sigma}(a)
 =3x^2+7u^2+7sp+ps
 =3x^2+7u^2+8ps.
\]
The remaining fourteen parameters do not occur, yielding the displayed orthogonal decomposition.
\end{proof}

These formulas hold in every characteristic different from $2$, with the integer coefficients interpreted in $K$. If $\operatorname{char}K\notin\{2,3,7\}$, the radical of $q_{A(\cQ),\sigma}$ equals the Jacobson radical of $A(\cQ)$. In characteristic $3$ or $7$, the radical of the quadratic form has dimension $15$, while the Jacobson radical still has dimension $14$.

\begin{corollary}[Trace-form obstruction]
\label{cor:trace-obstruction}
If $K$ is formally real, then $\cQ$ is not weakly hyperbolic.
\end{corollary}

\begin{proof}
A formally real field has characteristic $0$. At every ordering $P$ of $K$, the one-dimensional forms $\langle3\rangle$ and $\langle7\rangle$ are positive definite, while the middle two-dimensional form in \eqref{eq:trace-form-explicit} is hyperbolic. Hence
\[
 \sgnop_P q_{A(\cQ),\sigma}=2.
\]
First's general criterion, \cref{thm:first-A}, shows that $\cQ$ is not weakly hyperbolic.
\end{proof}

\section*{Declarations}

\noindent\textbf{Funding.}
This work was supported by the National Natural Science Foundation of China (Grant Nos.~12231009 and 11971224).

\medskip
\noindent\textbf{Competing interests.}
The author declares no competing interests.

\medskip
\noindent\textbf{Data availability.}
No datasets were generated or analyzed in this study.


\begin{thebibliography}{9}

\bibitem{BFFM2014}
E. Bayer-Fluckiger, U.A. First, D.A. Moldovan,
Hermitian categories, extension of scalars and systems of sesquilinear forms,
Pacific J. Math. 270 (1) (2014) 1--26.
\url{https://doi.org/10.2140/pjm.2014.270.1}.

\bibitem{First2020}
U.A. First,
Pfister's local--global principle and systems of quadratic forms,
Bull. Lond. Math. Soc. 52 (6) (2020) 1105--1121.
\url{https://doi.org/10.1112/blms.12385}.

\bibitem{Lam2005}
T.Y. Lam,
Introduction to Quadratic Forms over Fields,
Graduate Studies in Mathematics, vol.~67,
American Mathematical Society, Providence, RI, 2005.
\url{https://doi.org/10.1090/gsm/067}.

\bibitem{LeepSchueller1999}
D.B. Leep, L.M. Schueller,
Classification of pairs of symmetric and alternating bilinear forms,
Expo. Math. 17 (5) (1999) 385--414.

\bibitem{QSS1979}
H.-G. Quebbemann, W. Scharlau, M. Schulte,
Quadratic and Hermitian forms in additive and abelian categories,
J. Algebra 59 (2) (1979) 264--289.
\url{https://doi.org/10.1016/0021-8693(79)90126-1}.

\bibitem{Waterhouse1976}
W.C. Waterhouse,
Pairs of quadratic forms,
Invent. Math. 37 (2) (1976) 157--164.
\url{https://doi.org/10.1007/BF01418967}.

\bibitem{Wilson2013}
J.B. Wilson,
Division, adjoints, and dualities of bilinear maps,
Comm. Algebra 41 (11) (2013) 3989--4008.
\url{https://doi.org/10.1080/00927872.2012.660668}.

\end{thebibliography}
\end{document}